\documentclass[11pt]{article}
\usepackage{amsmath, amssymb, amsthm, color, epsfig, epstopdf, enumerate,framed}
\usepackage{graphicx, hyperref, listings, mathtools, multicol, multirow, pdfpages, subfigure}
\usepackage{booktabs, siunitx} % scientific tables
\usepackage{amsopn, authblk} % afflication
\usepackage[table]{xcolor}
\usepackage[numbers,sort&compress]{natbib} % compressed citation

\usepackage{algorithm, algpseudocode} % algorithm

\newcommand{\N}{\mathbb{N}} 
\newcommand{\R}{\mathbb{R}}
\newcommand{\argmin}[1]{\underset{#1}{\mathrm{argmin}}}
\newcommand{\card}{{\rm card}}

\newcommand{\rank}{{\rm rank}}
\newcommand{\sign}{{\rm sgn}}
\newcommand{\snr}{{\rm SNR}}
\newcommand{\supp}{{\rm supp}}

\theoremstyle{plain}
\newtheorem{theorem}{Theorem}[section]
\newtheorem*{theorem*}{Theorem}

\newtheorem{proposition}[theorem]{Proposition}
\newtheorem{corollary}[theorem]{Corollary}
\theoremstyle{definition}
\newtheorem{example}[theorem]{Example}

\theoremstyle{remark}

\newtheorem{remark}[theorem]{Remark}

\numberwithin{equation}{section}
\numberwithin{algorithm}{section}
\numberwithin{figure}{section}
\numberwithin{table}{section}

\title{On the Convergence of the SINDy Algorithm}
\author{Linan Zhang}
\author{Hayden Schaeffer}
\affil{Department of Mathematical Sciences, Carnegie Mellon University, Pittsburgh, PA 15213. (\text{linanz@andrew.cmu.edu}, { }\text{schaeffer@cmu.edu})}

\begin{document}
\maketitle

\begin{abstract}
One way to understand time-series data is to identify the underlying dynamical system which generates it. This task can be done by selecting an appropriate model and a set of parameters which best fits the dynamics while providing the simplest representation (\textit{i.e.} the smallest amount of terms). One such approach 
is the \textit{sparse identification of nonlinear dynamics} framework \cite{BruntonProctorKutz16PNAS} which uses a sparsity-promoting algorithm that iterates between a partial least-squares fit and a thresholding (sparsity-promoting) step. In this work, we provide some theoretical results on the behavior and convergence of the algorithm proposed in \cite{BruntonProctorKutz16PNAS}. In particular, we prove that the algorithm approximates local minimizers of an unconstrained $\ell^0$-penalized least-squares problem. From this, we provide sufficient conditions for general convergence, rate of convergence,  and conditions for one-step recovery. Examples illustrate that the rates of convergence are sharp. In addition, our results extend to other algorithms related to the algorithm in  \cite{BruntonProctorKutz16PNAS}, and provide theoretical verification to several observed phenomena.  \end{abstract}

\section{Introduction}
\label{sec: intro}

Dynamic model identification arises in a variety of fields, where one would like to learn the underlying equations governing the evolution of some given time-series data $u(t)$. This is often done by learning a first-order differential equation $\dot{u}=f(u)$ which provides a reasonable model for the dynamics. The function $f$ is unknown and must be learned from the data. Some applications including, weather modeling and prediction, development and design of aircraft, modeling the spread of disease over time, trend predictions, \textit{etc.}

Several analytical and numerical approaches have been developed to solve various model identification problems. One important contribution to model identification is \cite{BongardLipson07, SchmidtLipson09}, where the authors introduced a symbolic regression algorithm to determine underlying physical equations, like equations of motion or energies, from data. The key idea is to learn the governing equation directly from the data by fitting the derivatives with candidate functions while balancing between accuracy and parsimony. In \cite{BruntonProctorKutz16PNAS}, the authors proposed the \textit{sparse identification of nonlinear dynamics (SINDy) algorithm}, which computes sparse solutions to linear systems related to model identification and parameter estimation. The main idea is to convert the (nonlinear) model identification problem to a linear system:
\begin{align}
Ax = b, \label{eq: sparse problem}
\end{align}
where the matrix $A\in\R^{m\times n}$ is a data-driven dictionary whose columns are (nonlinear) candidate functions of the given data $u$, the unknown vector $x\in\R^n$ represents the coefficients of the selected terms in the governing equation $f$, and the vector $b\in\R^m$ is an approximation to the first-order time derivative $\dot{u}$. In this method, the number of candidate functions are fixed, and thus one assumes that the set of candidate functions is sufficiently large to capture the nonlinear dynamics present in the data. In order to select an accurate model (from the set of candidate functions) which does not overfit the data, the authors of \cite{BruntonProctorKutz16PNAS} proposed a sparsity-promoting algorithm. In particular, a sparse vector $x$ which approximately solves Equation \eqref{eq: sparse problem} is generated by the following iterative scheme:
\begin{subequations}\label{eq: sparse representation algorithm}
\begin{align}
S^k &= \{1\le j\le n: |x_j^k| \ge \lambda \}, \\
x^{k+1} &= \argmin{x\in\R^n:\ \supp(x)\subseteq S^k} \|Ax-b\|_2,
\end{align}
\end{subequations}
where $\lambda>0$ is a thresholding parameter and $\supp(x)$ is the support set of $x$. In practice, it was observed that the algorithm converged within a few steps and produced an appropriate sparse approximation to Equation~\eqref{eq: sparse problem}. These observations are quantified in our work.

There are several approaches which leverage sparse approximations for model identification. In \cite{KaiserKutzBrunton17},  the authors combined the SINDy framework with model predictive control to solve model identification problems given noisy data. The resulting algorithm is able to control nonlinear systems and identify models in real-time. 
In \cite{ManganKutzBruntonProctor17}, the authors introduced information criteria to the SINDy framework, where they selected the optimal model (with respect to the chosen information criteria) over various values of the thresholding parameter. Other approaches have been developed based on the SINDy framework, including: SINDy for rational governing equations \cite{ManganBruntonProctoKutz16}, SINDy with control \cite{BruntonProctorKutz16IFAC}, and SINDy for abrupt changes \cite{QuadeAbelKutzBrunton18}.  

In \cite{SchaefferMcCalla17}, a sparse regression approach for identifying dynamical systems via the weak form was proposed. The authors used the following constrained minimization problem:
\begin{align*}
\min_x \|x\|_0 \quad \text{subject to } \|Ax-b\|_2\le \sigma, 
\end{align*}
where the dictionary matrix $A$ is formulated using an integrated set of candidate functions. In \cite{SchaefferTranWard17jul}, several sampling strategies were developed for learning dynamical equations from high-dimensional data. It was proven analytically and verified numerically that under certain conditions, the underlying equation can be recovered exactly from the following constrained minimization problem, even when the data is under-sampled: 
\begin{align*}
\min_x \|x\|_1 \quad \text{subject to } \|Ax-b\|_2\le \sigma, 
\end{align*}
where the dictionary matrix $A$ consists of second-order Legendre polynomials applied to the data. In \cite{SchaefferTranWard17sep}, the authors developed an algorithm for learning dynamics from multiple time-series data, whose governing equations have the same form but different (unknown) parameters. The authors provided convergence guarantees for their group-sparse hard thresholding pursuit algorithm; in particular, one can recover the dynamics when the data-drive dictionary is coercive. In \cite{TranWard17}, the authors provided conditions for exact recovery of dynamics from highly corrupted data generated by Lorenz-like systems. To separate the corrupted data from the uncorrupt points, they solve the minimization problem:
\begin{align*}
\min_{x,\eta} \|\eta\|_{2,1} \quad \text{subject to } Ax + \eta = b \text{ and $x$ is sparse},
\end{align*}
where the residual $\eta:=b-Ax$ is the variable representing the (unknown) corrupted locations. When the data is  a function of both time and space, {\it i.e.} $u=u(t,y)$ for some spatial variable $y$, the dictionary can incorporate spatial derivatives \cite{Schaeffer17, RudyBruntonProctorKutz17}. In  \cite{Schaeffer17}, an unconstrained $\ell^1$-regularized least-squares problem (LASSO \cite{Tibshirani96}) with a dictionary build from nonlinear functions of the data and its spatial partial derivatives was used to discover PDE from data. In \cite{RudyBruntonProctorKutz17}, the authors proposed an adaptive SINDy algorithm for discovering PDE, which iteratively applies ridge regression with hard thresholding. Additional approaches for model identification can be found in \cite{SorokinaSygletosTuritsyn16, BoninsegnaNuskeClementi17, DamBronsRasmussenNaulinHesthaven17, LongLuMaDong17, PantazisTsamardinos17, LoiseauBrunton18, RaissiKarniadakis18, SchaefferTranWardZhang2018}.

The sparse model identification approaches include a sparsity-promoting substep, typically through various thresholding operations. In particular, the SINDy algorithm, {\it i.e.} Equation \eqref{eq: sparse representation algorithm}, alternates between a reduced least-squares problem and a thresholding step. This is related to, but differs from, the iterative thresholding methods widely used in compressive sensing. To find a sparse representation of $x$ in Problem \eqref{eq: sparse problem}, it is natural to solve the $\ell^0$-minimization problems, where the $\ell^0$-penalty of a vector measures the number of its nonzero elements. In \cite{BlumensathYaghoobiDavies07, BlumensathDavies08, BlumensathDavies09}, the authors provided iterative schemes to solve the unconstrained and constrained $\ell^0$-regularized problems:
\begin{align}
&\min_x \|Ax-b\|_2^2 + \lambda^2\|x\|_0, \label{eq: l0 unconstrained} \\
&\min_x \|Ax-b\|_2^2 \quad \text{subject to } \|x\|_0\le s, \label{eq: l0 constrained}
\end{align}
respectively, where $\|A\|_2=1$. To solve Problem \eqref{eq: l0 unconstrained}, one iterates: 
\begin{align}
x^{k+1} = H_{\lambda}(x^k + A^T(b-Ax^k)), \label{eq: IHT unconstrained}
\end{align}
where $H_{\lambda}$ is the hard thresholding operator defined component-wise by:
\begin{align*}
H_{\lambda}(x)_j := \sign(x_j)\max(|x_j|,\lambda).
\end{align*}
To solve Problem \eqref{eq: l0 constrained}, one iterates: 
\begin{align}
x^{k+1} = L_s(x^k + A^T(b-Ax^k)), \label{eq: IHT constrained}
\end{align}
where $L_s$ is a nonlinear operator that only retains $s$ elements of $x$ with the largest magnitude and sets the remaining $n-s$ elements to zero.

The authors of \cite{BlumensathYaghoobiDavies07, BlumensathDavies08, BlumensathDavies09} also proved that the iterative algorithms defined by Equations \eqref{eq: IHT unconstrained} and \eqref{eq: IHT constrained} converge to the local minimizers of Problems \eqref{eq: l0 unconstrained} and \eqref{eq: l0 constrained}, respectively, and derived theoretically the error bounds and convergence rates for the solutions obtained via Equation \eqref{eq: IHT unconstrained}. In this work, we will show that the SINDy algorithm also finds local minimizers of Equation \eqref{eq: l0 unconstrained} and has similar theoretical guarantees.

\subsection{Contribution}

In this work, we show (in Section \ref{sec: analysis}) that the SINDy algorithm proposed in \cite{BruntonProctorKutz16PNAS} approximates the local minimizers of Problem \eqref{eq: l0 unconstrained}. We provide sufficient conditions for convergence and bounds on rate of convergence. We also prove that the algorithm typically converges to a local minimizer rapidly (in a finite number of steps). Based on several examples, the rate of convergence is sharp. We also show that the convergence results can be adapted to other SINDy-based algorithms. In Section~\ref{sec: application}, we highlight some of the theoretical results by applying the algorithm from \cite{BruntonProctorKutz16PNAS} to identify dynamical systems from noisy measurements.

\section{Convergence Analysis}
\label{sec: analysis}

Before detailing the results, we briefly introduce some notations and conventions. For an integer $n\in\N$, let $[n]\subset\N$ be the set defined by: $[n] := \{1,2,\cdots,n\}$.  Let $A$ be a matrix in $\R^{m\times n}$, where $m\ge n$. If $A$ is injective (or equivalently if $A$ is full column rank, \textit{i.e.} $\rank(A)=n$), then its pseudo-inverse $A^\dagger\in\R^{n\times m}$ is defined as $A^\dagger := (A^TA)^{-1}A^T$. Let $x\in\R^n$ and define the support set of $x$ as the set of indices corresponding to its nonzero elements, {\it i.e.},
\begin{align*}
\supp(x):=\{j\in[n]:x_j\ne0\}.
\end{align*}
The $\ell^0$ penalty of $x$ measures the number of nonzero elements in the vector and is defined as:
\begin{align*}
\|x\|_0:=\card(\supp(x)).
\end{align*}
The vector $x$ is called $s$-sparse if it has at most $s$ nonzero elements, thus $\|x\|_0\le s$.

Given a set $S\subseteq[n]$, where $n\in\N$ is known from the context, define $\bar{S}:=[n]\backslash S$.  For a matrix $A\in\R^{m\times n}$ and a set $S\subseteq[n]$, we denote by $A_S$ the submatrix of A in $\R^{m\times s}$ which consists of the columns of $A$ with indices $j\in S$, where $s=\card(S)$. Similarly, for a vector $x=\begin{pmatrix} x_1, x_2, \cdots, x_n \end{pmatrix}^T$, let $x_S$ be the subvector of $x$ in $\R^s$ consisting of the elements of $x$ with indices $j\in S$, or the vector in $\R^n$ which coincides with $x$ on $S$ and is zero outside $S$:
\begin{align*}
(x_S)_j =
\begin{cases} x_j &\quad \text{if }j\in S, \\
0 &\quad \text{if }j\in\bar{S}. \end{cases}
\end{align*}
The representation of $x_S$ should be clear within the context.

\subsection{Algorithmic Convergence}
Let $A\in\R^{m\times n}$ be a matrix with $m\ge n$ and $\rank(A)=n$, $x\in\R^n$ be the unknown signal, and $b\in\R^m$ be the observed data. The results presented here work for general $A$ satisfying these assumptions, but the specific application of interest is detailed in Section \ref{sec: application}. The SINDy algorithm from \cite{BruntonProctorKutz16PNAS} is:
\begin{subequations}\label{eq: iterative scheme}
\begin{align}
x^0 &= A^\dagger b, \label{eq: update x0} \\
S^k &= \{j\in[n]: |x_j^k| \ge \lambda \}, \quad k\ge 0, \label{eq: update sk} \\
x^{k+1} &= \argmin{x\in\R^n: \supp(x)\subseteq S^k} \|Ax-b\|_2 \quad k\ge 0. \label{eq: update xk}
\end{align}
\end{subequations}
which is used to find a sparse approximation to the solution of $Ax=b$. The following theorem shows that the SINDy algorithm terminates in finite steps.

\begin{theorem} \label{thm: n steps}
The iterative scheme defined by Equation \eqref{eq: iterative scheme} converges in at most $n$ steps.
\end{theorem}

\begin{proof}
Let $x^k$ be the sequence generated by Equation \eqref{eq: iterative scheme}. By Equation \eqref{eq: update xk} we have $\supp(x^{k+1})\subseteq S^k$, and from Equation \eqref{eq: update sk} we have $S^{k+1}\subseteq \supp(x^{k+1})$. Therefore, the sets $S^k$ are nested:
\begin{align}
& S^{k+1} \subseteq \supp(x^{k+1}) \subseteq S^k. \label{eq: nested sk} 
\end{align}
Consider the following two cases. If there exists an integer $M\in\N$ such that $S^{M+1} = S^M$, then:
\begin{align}
x^{M+2} &= \argmin{ \supp(x)\subseteq S^{M+1}} \|Ax-b\|_2 = \argmin{\supp(x)\subseteq S^M} \|Ax-b\|_2 = x^{M+1}. \label{eq: stopping criterion}
\end{align}
Thus, $x^k =x^{M+1}$ for all $k \ge M+1$, and $S^k=S^M$ for all $k\ge M$. Since $\card(S^k) \le n$ for all $k\in\N$, we conclude that $M\le n$, so that the scheme converges in at most $n$ steps.

On the other hand, if there does not exist an integer $M$ such that $S^{M+1} = S^M$ and $S^M\ne\emptyset$, then we have a sequence of strictly nested sets, {\it i.e.}, 
\begin{align*}
S^{k+1} \subsetneq S^k \quad \text{for all } k \text{ such that } S^k\ne\emptyset.
\end{align*}
Since $\card(S^k) \le n$ for all $k\in\N$, we must have $S^k=\emptyset$ for all $k>n$. Therefore, the scheme converges to the trivial solution within $n$ steps.
\end{proof}

\begin{remark} \label{rem: stopping criterion}
Equation \eqref{eq: stopping criterion} suggests that an appropriate stopping criterion for the scheme is that the sets are stationary, \textit{i.e.} $S^{k}=S^{k-1}$.
\end{remark}

Note that, since the support sets are nested, the scheme will converge in at most $\card(S^0)$ steps. The following is an immediate consequence of Theorem \ref{thm: n steps}.

\begin{corollary} \label{cor: n-s steps}
The iterative scheme defined by Equation \eqref{eq: iterative scheme} converges to an $s$-sparse solution in at most $\card(S^0)-s$ steps.
\end{corollary}

\subsection{Convergence to the Local minimizers}

In this section, we will show that the iterative scheme defined by Equation \eqref{eq: iterative scheme} produces a minimizing sequence for a non-convex objective associated with sparse approximations. This will lead to a clearer characterization of the fixed-points of the iterative scheme.

Without loss of generality, assume in addition that $\|A\|_2=1$. We first show that the scheme converges to a local minimizer of the following (non-convex) objective function:
\begin{align}
F(x) := \|Ax-b\|_2^2 + \lambda^2\|x\|_0, \quad x\in\R^n. \label{eq: objective function}
\end{align}

\begin{theorem} \label{thm: energy decreases}
The iterates $x^k$ generated by Equation \eqref{eq: iterative scheme} strictly decreases the objective function unless the iterates are stationary.
\end{theorem}

\begin{proof}
Define the auxiliary variable:
\begin{align}
y^k := x^k_{S^k}, \quad k\in\N, \label{eq: update yk}
\end{align} 
which plays the role of an intermediate approximation.
In particular, we will relate $x^{k+1}$ and $y^k$. 

Observe that Equation \eqref{eq: iterative scheme} emits several useful properties. First, we have shown in Equation \eqref{eq: nested sk} that $S^{k+1} \subseteq \supp(x^{k+1}) \subseteq S^k$. Next, by Equation \eqref{eq: update xk}, $x^{k+1}$ is the least-squares solution over the set $S^k$. By considering the derivative of $\|Ax-b\|_2^2$ with respect to $x$, we obtain that:
\begin{align}
\left(A^T (Ax^{k+1}-b) \right)_{S^k}=0, \label{eq: A^TAx=A^Tb on S}
\end{align}
and the solution $x^{k+1}$ to the above equation satisfies $x^{k+1}_{S^k} = (A_{S^k})^\dagger b$.
To relate $x^{k+1}$ and $y^k$, note that Equations \eqref{eq: nested sk} and \eqref{eq: update yk} imply that: 
\begin{align}
\supp(x^{k+1})\subseteq \supp(y^k) = S^k \subseteq \supp(x^k). \label{eq: nested sk with yk}
\end{align}
By Equations \eqref{eq: update xk} and \eqref{eq: nested sk with yk}, we have:
\begin{align}
\|Ax^{k+1}-b\|_2 \leq \|Ay^{k}-b\|_2, \quad \text{and} \quad \|x^{k+1}\|_0\le\|y^k\|_0, \label{eq: xk yk l2 norm}
\end{align}
respectively.

To show that the objective function decreases, we use the optimization transfer technique as in \cite{BlumensathYaghoobiDavies07}. Define the surrogate function $G$ for $F$:
\begin{align}
G(x,y) := \|Ax-b\|_2^2 - \|A(x-y)\|_2^2 + \|x-y\|_2^2 + \lambda^2\|x\|_0, \quad x\in\R^n. \label{eq: surrogate function}
\end{align}
Since $\|A\|_2 = 1$, the term $- \|A(x-y)\|_2^2 + \|x-y\|_2^2$ is non-negative:
\begin{align*}
- \|A(x-y)\|_2^2 + \|x-y\|_2^2 \ge -\|A\|_2^2\|x-y\|_2^2 + \|x-y\|_2^2 = 0,
\end{align*}
and thus we have $G(x,y)\geq F(x)$ and $G(x,x)=F(x)$ for all $x,y\in\R^n$.

Define the matrix $B:=I-A^TA$. Since $A$ is injective (which is implied by $\rank(A)=n$) and $\|A\|_2 = 1$, we have that the eigenvalues of $B$ are in the interval $[0,1]$. Fixing the index $k\in\N$, from Equations \eqref{eq: objective function} and \eqref{eq: xk yk l2 norm}-\eqref{eq: surrogate function}, we have:
\begin{align*}
F(x^{k+1}) = \|Ax^{k+1}-b\|_2^2 + \lambda^2\|x^{k+1}\|_0 &\le \|Ax^{k+1}-b\|_2^2 + \lambda^2\|x^{k+1}\|_0 + \|x^k - y^k\|_B^2 \\
&\le \|Ay^k-b\|_2^2 + \lambda^2\|y^k\|_0 + \|x^k - y^k\|_B^2 = G(y^k,x^k).
\end{align*}
It remains to show that $G(y^k,x^k) \le G(x^k,x^k)$. By Equation \eqref{eq: update yk}, we have:
\begin{align*}
x^k-y^k = x^k-x^k_{S^k}=x^k_{\supp(x^k)}-x^k_{S^k\cap \supp(x^k)}=x^k_{\bar{S}^k\cap\supp(x^k)},
\end{align*}
where we included the intersection with $\supp(x^k)$ to emphasize that the difference is zero outside of the support set of $x^k$. Thus, the difference with respect to the surrogate function simplifies to:
\begin{align}
&G(y^k,x^k) - G(x^k,x^k) \nonumber \\
&\quad= \|Ay^k-b\|_2^2 - \|A(y^k-x^k)\|_2^2 + \|y^k-x^k\|_2^2 + \lambda^2\|y^k\|_0  - \|Ax^k-b\|_2^2 - \lambda^2\|x^k\|_0 \nonumber \\
&\quad=- 2 \langle b, Ay^k \rangle   + 2 \langle Ay^k, Ax^k \rangle -2\|Ax^k\|_2^2 + 2 \langle b, Ax^k \rangle + \|x^k-y^k\|_2^2 + \lambda^2 \left( \|y^k\|_0 - \|x^k\|_0 \right) \nonumber \\
&\quad= -2 \langle y^k-x^k, A^T(b-Ax^k) \rangle + \|x^k-y^k\|_2^2 + \lambda^2 \left( \|y^k\|_0 - \|x^k\|_0 \right) \nonumber \\
&\quad= -2 \langle x^k_{\bar{S}^k\cap\supp(x^k)}, A^T(Ax^k-b) \rangle + \|x^k_{\bar{S}^k\cap\supp(x^k)}\|_2^2 + \lambda^2 \left( \|y^k\|_0 - \|x^k\|_0 \right). \label{eq: g decreasing 1}
\end{align}
By Equations \eqref{eq: nested sk} and \eqref{eq: A^TAx=A^Tb on S}, we can observe that: 
\begin{align*}
\supp(x^k_{\bar{S}^k\cap\supp(x^k)})\subseteq \supp(x^k) \subseteq S^{k-1},
\quad \text{and} \quad
(A^T(Ax^k-b))_{S^{k-1}}=0,
\end{align*}
respectively, which together imply that:
\begin{align}
\langle x^k_{\bar{S}^k\cap\supp(x^k)}, A^T(Ax^k-b) \rangle = 0. \label{eq: g decreasing 2}
\end{align}
In addition, by Equation \eqref{eq: nested sk with yk}, we have:
\begin{align}
\card(S^k) - \card(\supp(x^k)) = -\card(\supp(x^k)\backslash S^k) = -\card(\bar{S}^k\cap\supp(x^k)). \label{eq: card sk xk}
\end{align}

Consider the following two cases. If $\bar{S}^k\cap\supp(x^k)=\emptyset$, then by Equation \eqref{eq: nested sk}, we must have $\supp(x^k)=S^k$, {\it i.e.} $x^k=x^{k+1}$. Therefore, $x^k$ is a fixed point, and $F(x^\ell)$ is stationary for all $\ell\ge k$. If $\bar{S}^k\cap\supp(x^k)\ne\emptyset$, then there exists an integer $j\in\bar{S}^k\cap\supp(x^k)$ such that $|x^k_j| < \lambda$, and thus: 
\begin{align}
\|x^k_{\bar{S}^k\cap\supp(x^k)}\|_2^2 < \lambda^2 \, \card(\bar{S}^k\cap\supp(x^k)). \label{eq: s-bar supp-x}
\end{align}
Thus, by Equations \eqref{eq: card sk xk} and \eqref{eq: s-bar supp-x}, provided that the iterates are not stationary, we have:
\begin{align}
&\|x^k_{\bar{S}^k\cap\supp(x^k)}\|_2^2 + \lambda^2 \left( \|y^k\|_0 - \|x^k\|_0 \right) \nonumber \\
&\quad= \|x^k_{\bar{S}^k\cap\supp(x^k)}\|_2^2 + \lambda^2\left( \card(S^k) - \card(\supp(x^k)) \right) \nonumber \\
&\quad< \lambda^2\card(\bar{S}^k\cap\supp(x^k)) + \lambda^2\left( \card(S^k) - \card(\supp(x^k))\right) = 0. \label{eq: g decreasing 3}
\end{align} 
Combining Equations \eqref{eq: g decreasing 1}, \eqref{eq: g decreasing 2}, and \eqref{eq: g decreasing 3} yields: 
\begin{align*}
G(y^k,x^k) - G(x^k,x^k) < 0.
\end{align*}
Therefore, for $k\in\N$ such that the $k$-iteration is not stationary, we have:
\begin{align*}
F(x^{k+1}) \le G(y^k,x^k) < G(x^k,x^k) = F(x^k),
\end{align*}
which completes the proof.
\end{proof}

In the following theorem, we show that the scheme converges to a fixed point, which is a local minimizer of the objective function $F$.

\begin{theorem} \label{thm: fixed point}
The iterates $x^k$ generated by Equation \eqref{eq: iterative scheme} converges to a fixed point of the iterative scheme defined by Equation \eqref{eq: iterative scheme}. A fixed point of the scheme is also a local minimizer of the objective function defined by Equation \eqref{eq: objective function}.
\end{theorem}

\begin{proof}
We first observe that
\begin{align}
\|Ax^k-b\|_2\le\|b\|_2 \label{eq: Ax-b bound}
\end{align}
for all $k\in\N$. This is an immediate consequence of Equation \eqref{eq: update xk}:
\begin{align*}
\|Ax^k-b\|_2 = \min_{x\in\R^n: \supp(x)\subseteq S^{k-1}} \|Ax-b\|_2 \le \|b\|_2,
\end{align*}
since the zero vector is in the feasible set. Next, we show that:
\begin{align}
\sum_{k=1}^\infty \|x^{k+1}-x^k\|_2^2 < \infty. \label{eq: xk summability}
\end{align}
Denote the smallest eigenvalue of $A^TA$ by $\lambda_0$. The assumption that $A$ has full column rank implies that $\lambda_0>0$. Thus, by the coercivity of $A^TA$, 
\begin{align}
\|x^{k+1}-x^k\|_2^2 \le \dfrac{1}{\lambda_0} \|A(x^{k+1}-x^k)\|_2^2. \label{eq: coercivity aa}
\end{align}
for all $k\in\N$.
For $k\in\N$, define subsets $W^k,V^k\subseteq\R^m$ by:
\begin{align*}
W^k&:= \{Ax: x\in\R^n, \ \ \supp(x)\subseteq S^k\}, \\
V^k &:= \{r\in\R^m: \langle r,y\rangle = 0 \ \ \forall y\in W^k\} = (W^k)^\perp .
\end{align*}
Fixing $M\in\N$ and $k\in[M]$ and setting $r:=b-Ax^k$, we have $(A^Tr)_{S^{k-1}}=0$ by Equation \eqref{eq: A^TAx=A^Tb on S}. For $x\in\R^n$ with $\supp(x)\subseteq S^{k-1}$, we have $\langle r,Ax \rangle = \langle A^Tr,x \rangle = 0$, which implies that $r\in V^{k-1}$. In addition, $A(x-x^k)\in W^{k-1}$ for all $x\in\R^n$ with $\supp(x)\subseteq S^{k-1}$. Thus,
\begin{align}
\|Ax-b\|_2^2
&= \|A(x-x^k)\|_2^2 - 2 \langle A(x-x^k), r \rangle + \|Ax^k-b\|_2^2 = \|A(x-x^k)\|_2^2 + \|Ax^k-b\|_2^2 \label{eq: Ax-b triangle inequality}
\end{align}
for all $x\in\R^n$ with $\supp(x)\subseteq S^{k-1}$. By Equations \eqref{eq: nested sk} and \eqref{eq: Ax-b triangle inequality},
\begin{align}
\|A(x^{k+1}-x^k)\|_2^2 = \|Ax^{k+1}-b\|_2^2 - \|Ax^k-b\|_2^2 \label{eq: axk difference}
\end{align}
for $k\in[M]$. Combining Equations \eqref{eq: Ax-b bound}, \eqref{eq: coercivity aa}, and \eqref{eq: axk difference} yields:
\begin{align*}
\sum_{k=1}^M \|x^{k+1}-x^k\|_2^2 &\le \dfrac{1}{\lambda_0}  \sum_{k=1}^M \|A(x^{k+1}-x^k)\|_2^2 \\
&= \dfrac{1}{\lambda_0}  \sum_{k=1}^M \left( \|Ax^{k+1}-b\|_2^2 - \|Ax^k-b\|_2^2 \right)\\
&\le \dfrac{1}{\lambda_0} \ \|Ax^{M+1}-b\|_2^2 \le \dfrac{1}{\lambda_0} \|b\|_2^2,
\end{align*}
and Equation \eqref{eq: xk summability} follows by sending $M\to\infty$. 

We now show that the iterates $x^k$ converge to a fixed point of the scheme. Since $\|x^{k+1}-x^{k}\|_2 \to 0$ as $k\to\infty$, for any $\epsilon>0$, there exists an integer $N\in\N$ such that $\|x^{k+1}-x^{k}\| <\epsilon$ for all $k\ge N$. Assume to the contrary that the scheme does not converge. Then there exists an integer $K\ge N$ such that $S^{K} \backslash S^{K+1}\ne\emptyset$. Thus, we can find an index $j\in S^{K} \backslash S^{K+1}$. By Equation \eqref{eq: iterative scheme}, we must have $|x_j^K|\geq \lambda$, $|x_j^{K+1}|< \lambda$, and $x_j^{K+2}= 0$. Thus,
\begin{align*}
\lambda-|x_j^{K+1}| \le |x_j^{K}-x_j^{K+1}| \leq \|x^{K+1}-x^{K}\|_2  < \epsilon,
\end{align*}
and
\begin{align*}
|x_j^{K+1}| = |x_j^{K+1}-x_j^{K+2}| \leq \|x^{K+2}-x^{K+1}\|_2  < \epsilon.
\end{align*}
The two conditions on $|x_j^{K+1}|$ above implies that $\lambda-\epsilon< |x_j^{K+1}| < \epsilon$, which fails when, for example, $\epsilon = \lambda/3$. Therefore, the iterates $x^k$ converges.
In particular, the preceding argument indicates that there exists an integer $N\in\N$ such that $S^k\backslash S^{k+1}=\emptyset$ for all $k\ge N$. Since the sets $S^k$ are nested, we conclude that $S^{k+1}=S^k$ for all $k\ge N$. Therefore, the iterates $x^k$ converge to a fixed point of the scheme defined by Equation \eqref{eq: iterative scheme}.

We now show that a fixed point of the scheme is a local minimizer of the objective function defined by Equation \eqref{eq: objective function}. Let $x^*$ be a fixed point of the scheme. Then $x^*$ and the set $S^*:=\supp(x^*)$ satisfy:
\begin{align} \label{eq: fixed point}
S^* = \{j\in[n]: |x_j^*| \ge \lambda \} \quad \text{and} \quad x^* = \argmin{\supp(x)\subseteq S^*} \|Ax-b\|_2.
\end{align}
From Equation \eqref{eq: fixed point}, we observe that:
\begin{align}
(A^T(Ax^*-b))_{S^*} = 0, \label{eq: fp ob support set} 
\end{align}
and
\begin{align}
x^*_j\ne0 \iff |x^*_j|\ge\lambda. \label{eq: fp ob threshold}
\end{align}
To show that $x^*$ is a local minimizer of $F$, we will find a positive real number $\epsilon>0$ such that
\begin{align}
F(x^*+z) \ge F(x) \quad \text{for all } z\in\R^n \text{ with } \|z\|_\infty<\epsilon. \label{eq: local min}
\end{align}
Let $U\subseteq [n]$ be the complement of the support set of $x^*$:
\begin{align*}
U := \{j\in[n]: x^*_j=0\}.
\end{align*}
Then by Equation \eqref{eq: fp ob threshold},
\begin{align}
\bar{U} = \supp(x^*) = \{j\in[n]: x^*_j\ne0\} = \{j\in[n]: |x^*_j|\ge\lambda\} = S^*. \label{eq: u-bar s-star}
\end{align}
Fixing $z\in\R^n$, from Equation \eqref{eq: surrogate function}, we have:
\begin{align*}
G(x^*+z,x^*) - G(x^*,x^*) = 2 \langle Az, Ax^*-b \rangle +  \lambda^2 \left( \|x^*+z\|_0 - \|x^*\|_0 \right) + \|z\|_2^2,
\end{align*}
Let $a_j$ be the $j$-th column of $A$, then:
\begin{align}
&2 \langle Az, Ax^*-b \rangle +  \lambda^2 \left( \|x^*+z\|_0 - \|x^*\|_0 \right) \nonumber \\
&\quad= \sum_{j\in U} \left(2a_j^T(Ax^*-b)\, z_j + \lambda^2 |z_j|^0\right) + \sum_{j\in \bar{U}} \left(2a_j^T(Ax^*-b)\, z_j + \lambda^2 (|x^*_j+ z_j|^0 - |x^*_j|^0)\right) \nonumber\\
&\quad= \sum_{j\in U} \left(2a_j^T(Ax^*-b)\, z_j + \lambda^2 |z_j|^0\right) + \sum_{j\in \bar{U}}\lambda^2 (|x^*_j+ z_j|^0 - |x^*_j|^0),
\label{eq:sums}
\end{align}
where the last step follows from Equation \eqref{eq: fp ob support set}. To find an $\epsilon>0$ such that Equation \eqref{eq: local min} holds, we will show that Equation~\eqref{eq:sums} is non-negative (so that the difference in $G$ is bounded below by $\|z\|_2^2$).

 For $j\in\bar{U}$, we have $|x^*_j|\ge\lambda$ by Equation \eqref{eq: u-bar s-star}. If $|z_j|<\lambda$ for $j\in \bar{U}$, then $x^*_j+ z_j\ne0$, and thus $|x^*_j+ z_j|^0 - |x^*_j|^0 =0$. Therefore, provided that $|z_j|<\lambda$ for all $j\in \bar{U}$,
\begin{align*}
2 \langle Az, Ax^*-b \rangle +  \lambda^2 \left( \|x^*+z\|_0 - \|x^*\|_0 \right) = \sum_{j\in U} \left(2a_j^T(Ax^*-b)z_j + \lambda^2 |z_j|^0\right).
\end{align*}
For $j\in U$, consider the following two cases. If $z_j = 0$, then the term in the sum is zero:
\begin{align*}
2a_j^T(Ax^*-b)z_j + \lambda^2 |z_j|^0 = 0.
\end{align*}
If $|z_j| > 0$ and $\lambda^2 \ge 2|a_j^T(Ax^*-b)z_j|$, then,
\begin{align*}
2a_j^T(Ax^*-b)z_j + \lambda^2 |z_j|^0 = 2a_j^T(Ax^*-b)z_j + \lambda^2 \ge 0.
\end{align*} 
Combining these results: if $\epsilon$  satisfies,
\begin{align*}
0<\epsilon \le \lambda^2\min\left\{\, \min_{j\in[n]}\dfrac{1}{2|a_j^T(Ax^*-b)|}\, ,\ \ 1 \right\}.
\end{align*}
then for any $z\in\R^n$ with $\|z\|_\infty<\epsilon$, we have:\begin{align*}
G(x^*+z,x^*) - G(x^*,x^*) \ge \|z\|_2^2,
\end{align*}
which then implies that:
\begin{align*}
F(x^*+z) &= G(x^*+z,x^*) + \|Az\|_2^2 - \|z\|_2^2 \ge G(x^*+z,x^*) - \|z\|_2^2 \ge G(x^*,x^*) = F(x^*),
\end{align*}
and the proof is complete.
\end{proof}

We state a sufficient condition for global minimizers of the objective function in the following theorem.

%%%%%

\begin{theorem} \label{thm: Tropp06}
{\rm (Theorem 12 from \cite{Tropp06})} 
Let $x^g$ be a global minimizer of the objective function. Define $U_g := \{j\in[n]: x^g_j=0\}$. Then,
\begin{subequations}\label{eq: global min condition}
\begin{align}
|a_j^T(Ax^g-b)|\le\lambda \quad &\text{for all } j\in U_g, \label{eq: global min condition1} \\
|x^g_j|\ge\lambda \text{ and } a_j^T(Ax^g-b)=0 \quad &\text{for all } j\in \bar{U}_g. \label{eq: global min condition2}
\end{align}
\end{subequations}
\end{theorem}

%%%%%

Theorems \ref{thm: fixed point} and \ref{thm: Tropp06} immediately imply the following result.

\begin{corollary} \label{cor: global min}
A global minimizer of the objective function defined by Equation \eqref{eq: objective function} is a fixed point of the iterative scheme defined by Equation \eqref{eq: iterative scheme}.
\end{corollary}

%%%%%

Theorem~\ref{thm: fixed point} shows that the iterative scheme converges to a local minimizer of the objective function, but it does not imply that the iterative scheme can obtain all local minima. However, by Corollary \ref{cor: global min}, the global minimizer is indeed obtainable. The following proposition provides a necessary and sufficient condition that the scheme terminates in one step, which is a consequence of Corollary \ref{cor: global min}.

%%%%%

\begin{proposition} \label{prop: one step condition}
Let $x^*\in\R^n$ be a vector which satisfies $Ax^*=b$ and $|x^*_j|\ge\lambda$ on $S:=\supp(x^*)$. A necessary and sufficient condition that $x^*$ can be recovered using the iterative scheme defined by Equation \eqref{eq: iterative scheme} in one step is:
\begin{align}
\min_{j\in S}|(A^\dagger b)_j| \ge \lambda > \max_{j\in \bar{S}}|(A^\dagger b)_j|. \label{eq: one step condition} 
\end{align}
\end{proposition}

\begin{proof} 
First, observe from the definitions of $x^0$ and $S^0$ that
\begin{align*}
S^0 = S \iff \{j\in[n]: |(A^\dagger b)_j| \ge \lambda\} = S \iff \min_{j\in S}|(A^\dagger b)_j| \ge \lambda > \max_{j\in \bar{S}}|(A^\dagger b)_j|.
\end{align*}

Assume that $x^*$ can be recovered via the scheme in one step, {\it i.e.}, $x^1=x^*$. By the definition of $S^1$, it follows that
\begin{align*}
S^1 := \{j\in[n]: |x_j^1| \ge \lambda\} = \{j\in[n]: |x^*_j| \ge \lambda\} = S.
\end{align*} 
By the stopping criterion (see Remark \ref{rem: stopping criterion}), we have $S^1=S^0$. Thus, $S^0=S$, which implies Equation \eqref{eq: one step condition}.

Assume that Equation \eqref{eq: one step condition} holds, {\it i.e.}, $S^0=S$. The assumption that $Ax^*=b$ implies that:
\begin{align*}
 \|Ax^*-b\|_2 = \min_{x\in\R^n:\supp(x)\subseteq S}\|Ax-b\|_2 \end{align*}
 since $\supp(x^*) \in S$ and the norm is zero. Since $A$ is injective, we have uniqueness and,
\begin{align*}
x^* = \argmin{x\in\R^n:\supp(x)\subseteq S}\|Ax-b\|_2= \argmin{x\in\R^n:\supp(x)\subseteq S^0}\|Ax-b\|_2 = x^1,
\end{align*}
{\it i.e.} $x^*$ can recovered via the scheme in one step.
\end{proof} 

%%%%%

We summarize all of the convergence results in the following theorem. The algorithm proposed in \cite{BruntonProctorKutz16PNAS} is summarized in Algorithm \ref{alg: iterative scheme}.

%%%%%

\begin{theorem} \label{thm: summary}
Assume that $m\ge n$. Let $A\in\R^{m\times n}$ with $\|A\|_2=1$, $b\in\R^m$, and $\lambda>0$. Let $x^k$ be the sequence generated by Equation \eqref{eq: iterative scheme}. Define the objective function $F$ by Equation \eqref{eq: objective function}. We have:
\begin{enumerate}[(i)]
\item $x^k$ converges to a fixed point of the iterative scheme defined by Equation \eqref{eq: iterative scheme} in at most $n$ steps;
\item a fixed point of the scheme is a local minimizer of $F$;
\item a global minimizer of $F$ is a fixed point of the scheme;
\item $x^k$ strictly decreases $F$ unless the iterates are stationary.
\end{enumerate}
\end{theorem}

%%%%%

\begin{algorithm} [t!]
\caption{The SINDy algorithm \cite{BruntonProctorKutz16PNAS} for $Ax=b$}
\label{alg: iterative scheme}
\begin{algorithmic}[1]
\Require $m\ge n$; $A\in\R^{m\times n}$ with $\rank(A)=n$; $b\in\R^m$.
\State Set $k=0$; Initialize $x^0 = A^\dagger b$ and $S^{-1}=\emptyset$
\State Set $S^k = \{j\in[n]: |x_j^k| \ge \lambda \}$; Choose $\lambda>0$ such that $S^0\ne\emptyset$; 
\While{$S^k\ne S^{k-1}$}
\State $x^{k+1} = \text{argmin} \|Ax-b\|_2$ such that $\supp(x)\subseteq S^k$;
\State $S^{k+1} = \{j\in[n]: |x_j^{k+1}| \ge \lambda \}$;
\State $k=k+1$;
\EndWhile
\State \Return $x^k$.
\end{algorithmic}
\end{algorithm}

%%%%%

The preceding convergence analysis for Algorithm \ref{alg: iterative scheme} can be readily adapted to a variety of SINDy based algorithms. For example, in \cite{RudyBruntonProctorKutz17}, the authors proposed the Sequential Threshold Ridge regression (STRidge) algorithm, to find a sparse approximation of the solution of $Ax=b$. Instead of minimizing the function $F$, the STRidge algorithm minimizes the following objective function:
\begin{align}
F_1(x) := \|Ax-b\|_2^2 + \gamma\|x\|_2^2 + \lambda^2\|x\|_0, \quad x\in\R^n \label{eq: ridge objective function}
\end{align}
by iterating:
\begin{subequations} \label{eq: ridge iterative scheme}
\begin{align}
x^0 &= A^\dagger b, \\
S^k &= \{j\in[n]: |x_j^k| \ge \lambda \}, k\ge0 \\
x^{k+1} &= \argmin{x\in\R^n: \supp(x)\subseteq S^k} \|Ax-b\|_2^2 + \gamma\|x\|_2^2. 
\end{align}
\end{subequations}
We assume that the parameter $\gamma>0$ is fixed. By defining:
\begin{align}
\tilde A := \begin{pmatrix} A \\ \gamma I \end{pmatrix} \in\R^{(m+n)\times n}, \quad \tilde b := \begin{pmatrix} b \\ 0 \end{pmatrix} \in\R^{m+n}, \label{eq: tilde A and b}
\end{align}
then $F_1(x) = \|\tilde Ax-\tilde b\|_2^2 + \lambda^2\|x\|_0$ is equivalent to the objective function of Algorithm \ref{alg: iterative scheme} with $\tilde A$ and $\tilde b$. We then obtain the following corollary.

%%%%%

\begin{corollary} \label{cor: summary}
Let $A\in\R^{m\times n}$, $b\in\R^m$ and $\gamma,\lambda>0$. Assume that $\|\tilde A\|_2=1$, where $\tilde A$ is defined by Equation \eqref{eq: tilde A and b}. Let $x^k$ be the sequence generated by Equation \eqref{eq: ridge iterative scheme}. Define the objective function $F_1$ by Equation \eqref{eq: ridge objective function}. We have:
\begin{enumerate}[(i)]
\item the iterates $x^k$ converge to a fixed point of the iterative scheme defined by Equation \eqref{eq: ridge iterative scheme};
\item a fixed point of the scheme is a local minimizer of $F_1$;
\item a global minimizer of $F_1$ is a fixed point of the scheme;
\item the iterates $x^k$ strictly decrease $F_1$ unless the iterates are stationary.
\end{enumerate}
\end{corollary}

%%%%%
Note that we no longer require $A$ to be injective, since concatenation with the identity matrix makes $\tilde{A}$ injection. 
\subsection{Examples and Sharpness}

We construct a few examples to highlight the effects of different choices of $\lambda>0$. In particular, we show that the scheme obtains nontrivial sparse approximations, give an example where the minimizer is obtained in one step, and provide an example in which the maximum number of steps ({\it i.e.}, $n-1$ steps) is required\footnote{The code is available on \url{https://github.com/linanzhang/SINDyConvergenceExamples}.}. In all examples, $A$ is injective.

%%%%%

\begin{example} \label{ex: example1}
Consider a lower-triangular matrix $A\in\R^{5\times5}$ given by:
\begin{align*}
A := \begin{pmatrix}
1 & 0 & 0 & 0 & 0 \\
-0.1 & 0.9 & 0 & 0 & 0 \\
-0.1 & -0.1 & 0.8 & 0 & 0 \\
-0.1 & -0.1 & -0.1 & 0.7 & 0 \\
-0.1 & -0.1 & -0.1 & -0.1 & 0.6
\end{pmatrix}. 
\end{align*}
Let $x,b\in\R^5$ be such that
\begin{align*}
x &:= \begin{pmatrix} 10, 0.95, 0.9, 0.85, 0.8 \end{pmatrix}^T, \\
b &:= Ax = \begin{pmatrix} 10, -0.145, -0.375, -0.59, -0.79\end{pmatrix}^T. 
\end{align*}
We want to obtain a 1-sparse approximation of the solution $x$ from the system $Ax=b$. First, observe that: 
\begin{align*}
\min_{j\in S}|(A^\dagger b)_j| = 10, \quad \max_{j\in \bar{S}}|(A^\dagger b)_j| = 0.95,
\end{align*}
where $S=\{1\}$. Thus by Proposition~\ref{prop: one step condition} choosing $\lambda\in(0.95,10]$ will yield immediate convergence:
\begin{subequations} \label{eq: example1 1step}
\begin{align}
x^0 &= \begin{pmatrix} 10, 0.95, 0.9, 0.85, 0.8\end{pmatrix}^T, & S^0 &= \{1\}, \\
x^1 &= \begin{pmatrix} 9.7981, 0, 0, 0, 0\end{pmatrix}^T, & S^1 &= \{1\}.
\end{align}
\end{subequations}
Indeed, we obtain a 1-sparse approximation of $x$ in one step. Now consider a parameter outside of the optimal range, for example $\lambda = 0.802$. Applying Algorithm \ref{alg: iterative scheme} to the linear system yields:
\begin{subequations} \label{eq: example1 nsteps}
\begin{align}
x^0 &= \begin{pmatrix} 10, 0.95, 0.9, 0.85, 0.8\end{pmatrix}^T, & S^0 &= \{1,2,3,4\}, \\
x^1 &= \begin{pmatrix} 9.9366, 0.8725, 0.8031, 0.7255, 0\end{pmatrix}^T & S^1 &= \{1,2,3\}, \\
x^2 &= \begin{pmatrix} 9.8869, 0.8117, 0.7271, 0, 0\end{pmatrix}^T, & S^2 &= \{1,2\}, \\
x^3 &= \begin{pmatrix} 9.8417, 0.7566, 0, 0, 0\end{pmatrix}^T, & S^3 &= \{1\}, \\
x^4 &= \begin{pmatrix} 9.7981, 0, 0, 0, 0\end{pmatrix}^T, & S^4 &= \{1\}.
\end{align}
\end{subequations}
Therefore, a 1-sparse approximation of $x$ is obtained in four steps, which is the maximum number of iterations Algorithm \ref{alg: iterative scheme} needs in order to obtain a 1-sparse approximation (see Corollary \ref{cor: n-s steps}). \qed
\end{example}

%%%%%

The following examples shows that the iterative scheme, Equation \eqref{eq: iterative scheme}, obtains fixed-points which are not obtainable via direct thresholding. In fact, if we re-order the support sets $S^k$ based on the magnitude of the corresponding components, we observe that the locations of the correct indices will evolve over time. The provides evidence that, in general, iterating the scheme is required.  
   
%%%%%

\begin{example} \label{ex: example2}
Consider the matrix $A\in\R^{10\times10}$ given by:
\begin{align*}
A = \begin{pmatrix}
4 & 5 & 1 & 6 & 8 & 4 & 6 & 6 & 2 & 7 \\
6 & 5 & 7 & 5 & 3 & 3 & 2 & 5 & 9 & 2 \\
1 & 5 & 1 & 7 & 4 & 8 & 1 & 3 & 9 & 7 \\
10 & 2 & 9 & 5 & 5 &10 & 0 & 8 & 1 & 2 \\
9 & 9 & 3 & 9 & 6 & 4 & 3 & 7 & 1 & 4 \\
10 & 1 & 7 & 8 & 7 & 4 &10 & 3 & 3 & 6 \\
2 & 4 & 4 & 5 & 6 & 9 & 1 & 9 & 1 & 9 \\
2 & 5 & 1 & 3 & 6 & 3 &10 & 7 & 2 & 1 \\
1 & 1 & 1 & 3 & 10 & 4 & 4 & 4 & 5 & 1 \\
6 & 5 & 1 & 4 & 2 & 5 & 1 & 5 & 1 & 8 \\
\end{pmatrix}. 
\end{align*}
Let $x,\eta,b\in\R^{10}$ be such that
\begin{align*}
x &:= \begin{pmatrix} 1,1,1,0,0,0,0,0,0,0\end{pmatrix}^T, \\
\eta &:= \begin{pmatrix} 0.23, 0.08, -0.01, -0.02, 0.04, -0.28, -0.32, 0.09, 0.30, 0.63\end{pmatrix}^T, \\
b &:= Ax+\eta = \begin{pmatrix} 10.23, 18.08, 6.99, 20.98, 21.04, 17.72, 9.68, 8.09, 3.30, 12.63\end{pmatrix}^T, 
\end{align*}
where each element of $\eta$ is drawn i.i.d. from the normal distribution $\mathcal{N}(0,0.25)$. We want to recover $x$ from the noisy data $b$ using Algorithm \ref{alg: iterative scheme}. The support set to be recovered is $S:=\{1,2,3\}$.
Setting $\lambda=0.7$ in Algorithm \ref{alg: iterative scheme} yields:
\begin{subequations}\label{eq: example2}
\begin{align}
x^0 &= \begin{pmatrix} \textcolor{blue}{0.88}, \textcolor{blue}{2.83}, \textcolor{blue}{2.04}, -1.60,  0.84, 0.63, 0.13, -1.82, -0.42, 0.26\end{pmatrix}^T, & S^0 &= \{2,3,8,4,1,5\}, \\
x^1 &= \begin{pmatrix} \textcolor{blue}{1.06}, \textcolor{blue}{1.08}, \textcolor{blue}{0.96}, -0.10, 0.04, 0, 0, -0.03, 0, 0\end{pmatrix}^T, & S^1 &= \{2,1,3\}, \\
x^2 &= \begin{pmatrix} \textcolor{blue}{1.04}, \textcolor{blue}{1.01}, \textcolor{blue}{0.94}, 0, 0, 0, 0, 0, 0, 0\end{pmatrix}^T, & S^2 &= \{1,2,3\},
\end{align}
\end{subequations}
where each $S^k$ is re-ordered such that the $j$-th element of $S^k$ is the $j$-th largest (in magnitude) element in $x^{k}$. Note that we have highlighted the desired components in blue.

Several important observations can be made from this example. First, there is no choice of $\lambda$ so that the method converges in one step, since the value of $x_1^0$ is smaller (in magnitude) than two components on $\bar{S}$; however, the method still terminates at the correct support set. Second, setting $\lambda>0.9$ will remove the first component immediately, yielding an incorrect solution.
Lastly, the order of the indices in the support set changes between steps, which shows that the solution $x^k$ is not simply generated by peeling offing the smallest elements of $A^\dagger b$. These observations lead one to conclude that the iterative scheme is more refined than just choosing the most important terms from $A^\dagger b$, \textit{i.e.} the iterations shuffle the components and help to locate the correct components.

 \qed
\end{example}

%%%%%

In the following example, we provide numerical support for Theorem \ref{thm: energy decreases}.

%%%%%

\begin{example}
In Table \ref{tab: objective function}, we list the values of the objective function $F$ for the different experiments, where $F$ is defined by Equation \eqref{eq: objective function}. Recall that in Theorem \ref{thm: energy decreases}, we have assumed that $\|A\|_2=1$. Thus to compute $F(x^k)$ for a given example, one may need to rescale the equation $Ax=b$ by $\|A\|_2$.  It can be seen from Table \ref{tab: objective function} that the value of $F(x^k)$ strictly decreases in $k$. \qed

\begin{table}[b!] 
\caption{The value of the objective function $F$ in different experiments.} 
\label{tab: objective function}
\centering \vskip5pt
\begin{tabular}{c | c | c | c } \toprule
Inputs $A$ and $b$ & Parameter $\lambda$ & Outputs $x^k$ and $S^k$ & $F(x^k)$ \\  \midrule
\multirow{2}{*}{as defined in Example \ref{ex: example1}} & \multirow{2}{*}{8} & \multirow{2}{*}{as given in Equation \eqref{eq: example1 1step}} & $F(x^0)=320.0000$ \\ 
& & & $F(x^1) = 65.2119$ \\ \midrule
\multirow{5}{*}{as defined in Example \ref{ex: example1}} & \multirow{5}{*}{0.802} & \multirow{5}{*}{as given in Equation \eqref{eq: example1 nsteps}} & $F(x^0) = 3.2160$ \\
& & & $F(x^1) = 2.7727$ \\
& & & $F(x^2) = 2.3688$ \\
& & & $F(x^3) = 2.0490$ \\
& & & $F(x^4) = 1.8551$ \\ \midrule 
\multirow{3}{*}{as defined in Example \ref{ex: example2}} & \multirow{3}{*}{0.7} & \multirow{3}{*}{as given in Equation \eqref{eq: example2}} & $F(x^0) = 4.9000$ \\
& & & $F(x^1) = 2.9401$ \\
& & & $F(x^2) = 1.4702$ \\ \bottomrule
\end{tabular}\end{table}

\end{example}

\section{Application: Model Identification of Dynamical Systems}
\label{sec: application}

Let $u$ be an observed dynamic process governed by a first-order system:
\begin{align*}
\dot{u}(t) = f(u(t)),
\end{align*}
where $f$ is an unknown nonlinear equation. One application of the SINDy algorithm is for the recovery (or approximation) of $f$ directly from data. In this section, we apply Algorithm \ref{alg: iterative scheme} to this problem and show that relatively accurate solutions can be obtained when the observed data is perturbed by a moderate amount of noise.

Before detailing the numerical experiments, we first define two relevant quantities used in our error analysis. Let $x\in\R^n$ be the (noise-free) coefficient vector and $\eta\in\R^n$ be the (mean-zero) noise. The signal-to-noise ratio (SNR) is defined by:
\begin{align*}
\snr(x,\eta) := 10\log_{10}\left(\dfrac{\|x-\text{mean}(x)\|_2^2}{\|\eta\|_2^2}\right).
\end{align*}
Given $Ax=b$, let $x_{\rm true}$ be the correct sparse solution that solves the noise-free linear system, and $x$ be the approximation of $x_{\rm true}$ returned by Algorithm \ref{alg: iterative scheme}. The relative error $E$ of $x$ is defined by:
\begin{align*}
E(x):= \dfrac{\|x-x_{\rm true}\|_2}{\|x_{\rm true}\|_2}.
\end{align*}

%%%%%

\subsection{The Lorenz System}

Consider the Lorenz system:
\begin{equation} \label{eq: lorenz system}
\begin{cases}
\dot{u}_1 = 10(u_2 - u_1), \\
\dot{u}_2 = u_1(28 - u_3) - u_2, \\
\dot{u}_3 = u_1u_2 - \frac{8}{3}u_3,
\end{cases}
\end{equation}
which produces chaotic solutions. To generate the synthetic data for this experiment, we set the initial data $u(0)=\begin{pmatrix} -5, 10, 30 \end{pmatrix}^T$ and evolve the system using the Runge-Kutta method of order 4 up to time-stamp $T = 10$ with time step $h=0.025$. The simulated data is defined as $u(t)$. The noisy data $\tilde u(t)$ is obtained by adding Gaussian noise directly to $u(t)$:
\begin{align*}
\tilde u = u + \eta, \quad \eta \sim \mathcal{N}(0,\sigma^2).
\end{align*}
Let $A=A(\tilde u(t))$ be the dictionary matrix consisting of (tensorized) polynomials in $\tilde u$ up to order $p$:
\begin{align*}
A = 
\begin{pmatrix}
| & | & | & | & & | \\
1 & P(\tilde u(t)) & P^2(\tilde u(t)) & P^3(\tilde u(t)) & \cdots & P^p(\tilde u(t)) \\
| & | & | & | & & |
\end{pmatrix},
\end{align*}
where
\begin{subequations} \label{eq: lorenz polynomial}
\begin{align}
P(\tilde u(t)) &:= 
\begin{pmatrix}
| & | & | \\ \tilde u_1(t) & \tilde u_2(t) & \tilde u_3(t) \\ | & | & |
\end{pmatrix}, \\
P^2(\tilde u(t)) &:= 
\begin{pmatrix}
| & | & | & | & | & | \\ \tilde u_1(t)^2 & \tilde u_1(t)\tilde u_2(t) & \tilde u_1(t)\tilde u_3(t) & \tilde u_2(t)^2 & \tilde u_2(t)\tilde u_3(t) & \tilde u_3(t)^2 \\ | & | & | & | & | & |
\end{pmatrix}, \\
P^3(\tilde u(t)) &:= 
\begin{pmatrix}
| & | & | & | &  & | \\
\tilde u_1(t)^3 & \tilde u_1(t)^2 \tilde u_2(t) & \tilde u_1(t)^2\tilde u_3(t) & \tilde u_1(t)\tilde u_2(t)^2 & \cdots & \tilde u_3(t)^3 \\
| & | & | & | &  & |
\end{pmatrix},
\end{align}
\end{subequations}
and so on. Each column of the matrices in Equation~\ref{eq: lorenz polynomial} is a particular polynomial (candidate function) and each row is a fixed time-stamp.  Let $b$ be the numerical approximation of $\dot u$:
\begin{align} \label{eq: lorenz derivative}
b_i(kh) := \begin{cases}
\dfrac{\tilde u_i(h)-\tilde u_i(0)}{h} &\quad\text{if }kh=0, \\[10pt]
\dfrac{\tilde u_i((k+1)h)-\tilde u_i((k-1)h)}{2h} &\quad\text{if }0<kh<T, \\[10pt]
\dfrac{\tilde u_i(T)-\tilde u_i(T-h)}{h} &\quad\text{if }kh=T,\\
\end{cases}
\end{align} 
for $i=1,2,3$. Note that $b$ is approximated directly from the noisy data, so it will be inaccurate (and likely unstable). We want to recover the governing equation for the Lorenz system ({\it i.e.}, the right-hand side of Equation \eqref{eq: lorenz system}) by finding a sparse approximation to solution of the linear system $Ax = b$ using Algorithm \ref{alg: iterative scheme}. 

With $p=5$ and $\lambda=0.8$, we apply Algorithm \ref{alg: iterative scheme} on data with different noise levels. The resulting approximations for $x$ (the coefficients) are listed in Table \ref{tab: lorenz coefficient}. The identified systems are:
\begin{enumerate}[(i)]
\item $\sigma^2=0.1$ (where $\snr(u,\eta)=41.1508$):
\begin{align} \label{eq: lorenz learned1}
\begin{cases}
\dot{u}_1 = -9.8122\,u_1 + 9.8163\,u_2 \\
\dot{u}_2 = 27.1441\,u_1 - 0.8893\,u_2 - 0.9733\,u_1\,u_3 \\
\dot{u}_3 = - 2.6238\,u_3 + 0.9841\,u_1\,u_2
\end{cases}
\end{align}
with $E(x)=0.0278$;
\item $\sigma^2=0.5$ (where $\snr(u,\eta)=27.0682$):
\begin{align} \label{eq: lorenz learned2}
\begin{cases}
\dot{u}_1 = -9.7012\,u_1 + 9.6980\,u_2 \\
\dot{u}_2 = 27.0504\,u_1 - 0.8485\,u_2 - 0.9717\,u_1\,u_3 \\
\dot{u}_3 = - 2.6197\,u_3 + 0.9834\,u_1\,u_2
\end{cases}
\end{align}
with $E(x)=0.0334$.
\end{enumerate}

\begin{table}[b!] 
\caption{\textbf{Lorenz System:} The recovered coefficients for two noise levels.} 
\label{tab: lorenz coefficient}
\centering \vskip5pt
\begin{tabular}{c | ccc | ccc} \toprule
\multirow{2}{*}{$A$} & \multicolumn{3}{c|}{$\sigma^2=0.1$} & \multicolumn{3}{c}{$\sigma^2=0.5$} \\ \cline{2-7}
& {$\dot{u}_1$} & {$\dot{u}_2$} & {$\dot{u}_3$} & {$\dot{u}_1$} & {$\dot{u}_2$} & {$\dot{u}_3$} \\ \toprule
{$1$} 	& 0	& 0	& 0	& 0	& 0	& 0  \\ \midrule
{$u_1$} 	& -9.8122	& 27.1441& 0	& -9.7012 & 27.0504 & 0 \\ \midrule
{$u_2$}	& 9.8163	& -0.8893 & 0	& 9.6980	& -0.8485	& 0 \\ \midrule
{$u_3$}	& 0	& 0	& -2.6238	& 0	& 0	& -2.6197 \\ \midrule
{$u_1^2$} & 0	& 0	& 0	& 0	& 0	& 0 \\ \midrule
{$u_1u_2$}& 0 	& 0	&0.9841	& 0	& 0	& 0.9834 \\ \midrule
{$u_1u_3$}& 0 	& -0.9733	& 0	& 0	& -0.9717	& 0 \\ \midrule
{$u_2^2$}	& 0	& 0	& 0	& 0	& 0	& 0 \\ \midrule
{$u_2u_3$}& 0	& 0	& 0	& 0	& 0	& 0 \\ \midrule
{$u_3^2$} & 0	& 0	& 0	& 0	& 0	& 0 \\ \midrule
{$u_1^3$} & 0	& 0	& 0	& 0	& 0	& 0 \\ \midrule
{$\vdots$} & {$\vdots$} & {$\vdots$} & {$\vdots$} & {$\vdots$} & {$\vdots$} &{$\vdots$}  \\ \midrule
{$u_3^5$} & 0	& 0	& 0	& 0	& 0	& 0 \\ \bottomrule
\end{tabular}\end{table}

To compare between the identified and true systems in the presence of additive noise on the observed data, we simulate the systems up to time-stamps $t = 20$ and $t=100$. The resulting trajectories are shown in Figures \ref{fig: lorenz plot2} and \ref{fig: lorenz plot3}.  

Figure 3.\ref{fig: lorenz1} shows that for a relatively small amount of noise ($\sigma^2=0.1$) the trajectory of the identified system almost coincide with the Lorenz attractor for a short time, specifically from $t=0$ to about $t=5$. On the other hand, Figure 3.\ref{fig: lorenz2} shows that for a larger amount of noise ($\sigma^2=0.5$) the error between the trajectories of the identified system and the Lorenz attractor remains small for a shorter time (up to about $t=4$). As expected, increasing the amount of noise will cause larger errors on the estimated parameters, and thus on the predicted trajectories. 
In both cases, the algorithm picks out the correct terms in the model. Increasing the noise will eventually lead to incorrect solutions.

\begin{figure}[h!]
\centering
\includegraphics[width = 6 in]{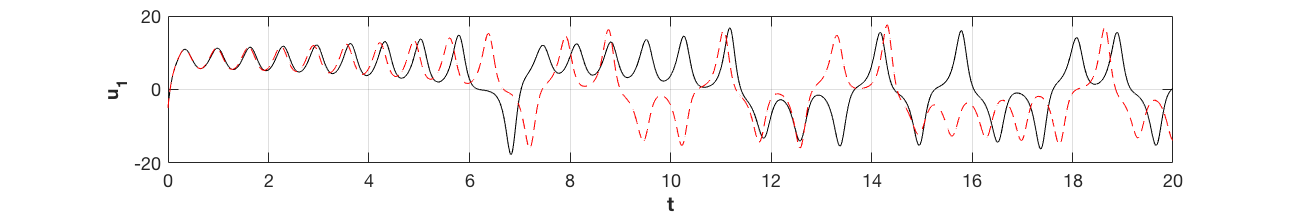}
\includegraphics[width = 6 in]{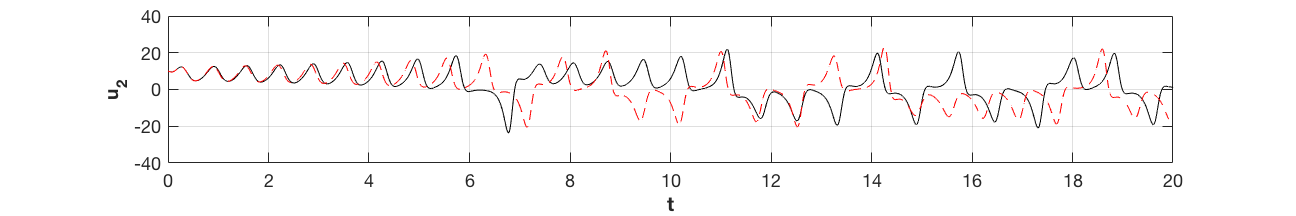}
\subfigure[$\sigma^2=0.1$]{\label{fig: lorenz1}\includegraphics[width = 6 in]{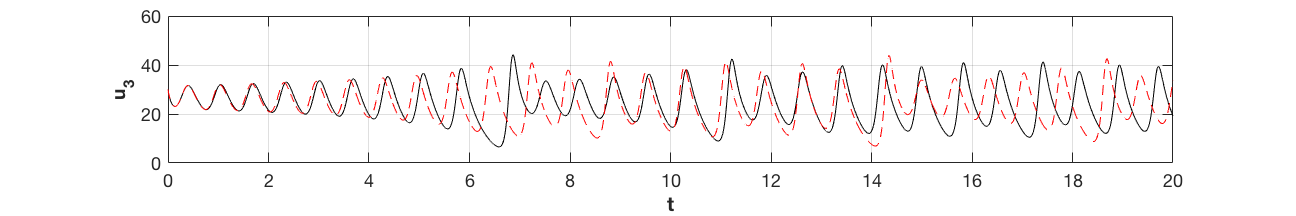}}
\includegraphics[width = 6 in]{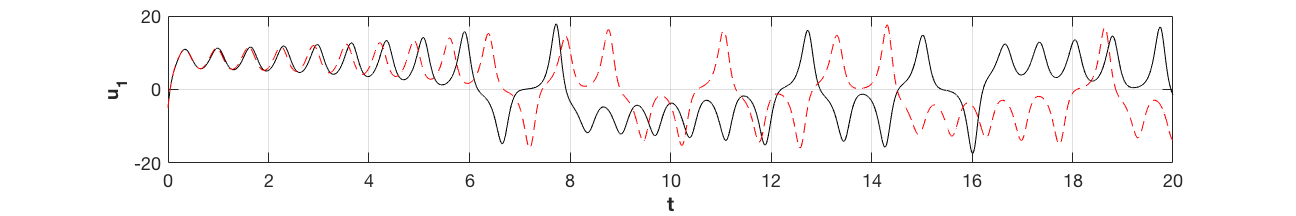}
\includegraphics[width = 6 in]{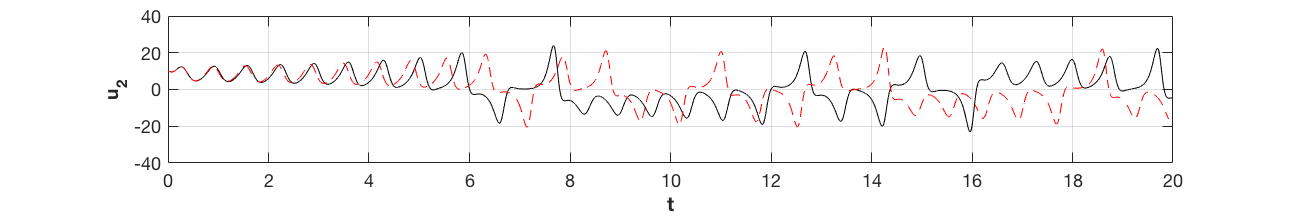}
\subfigure[$\sigma^2=0.5$]{\label{fig: lorenz2}\includegraphics[width = 6 in]{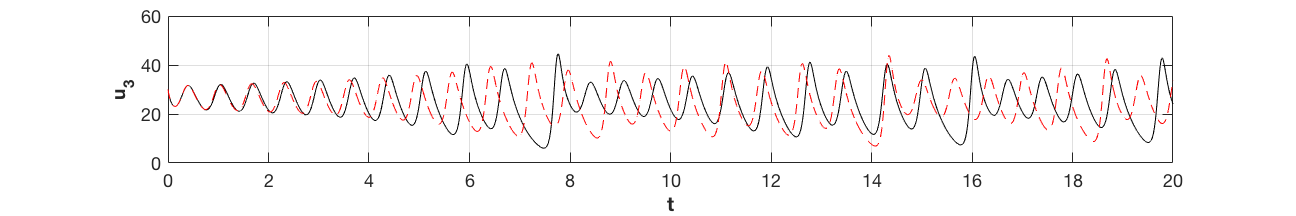}}
\caption{\textbf{Lorenz system}: Component-wise evolution of the trajectories. Solid line: the trajectory of the identified systems defined by: (a) Equation \eqref{eq: lorenz learned1} and (b) Equation \eqref{eq: lorenz learned2}, respectively. Red dashed line: the ``true'' Lorenz attractor.}
\label{fig: lorenz plot2}
\end{figure}

\begin{figure}[h!]
\centering
\subfigure[The Lorenz attractor]{\label{fig: lorenz exact}
	\includegraphics[width = 3 in]{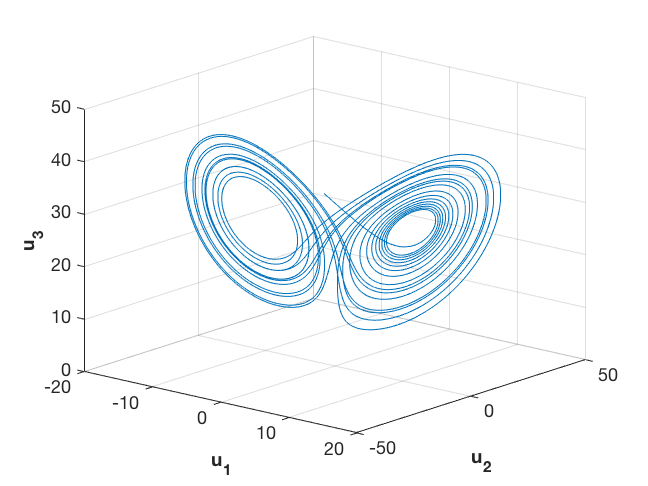}
	\includegraphics[width = 3 in]{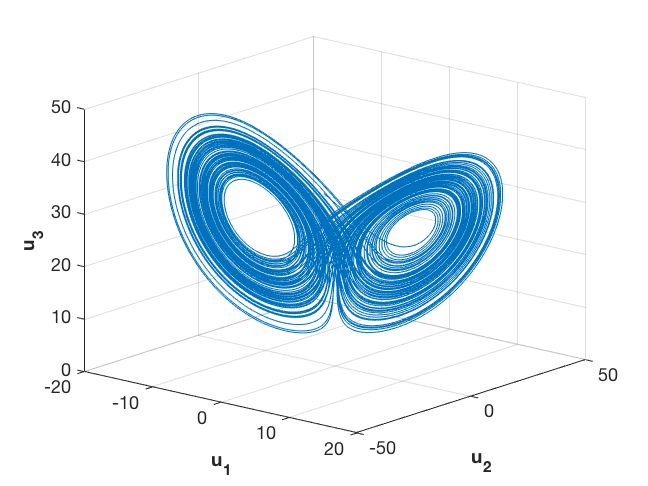}
	}
\subfigure[The trajectory defined by Equation \eqref{eq: lorenz learned1}]{
	\includegraphics[width = 3 in]{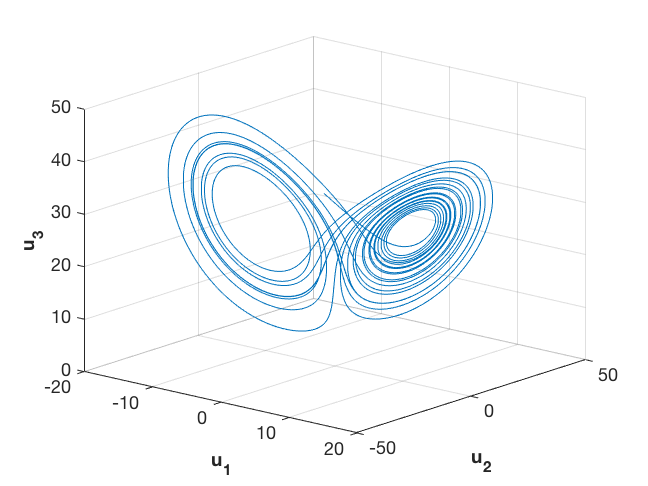}
	\includegraphics[width = 3 in]{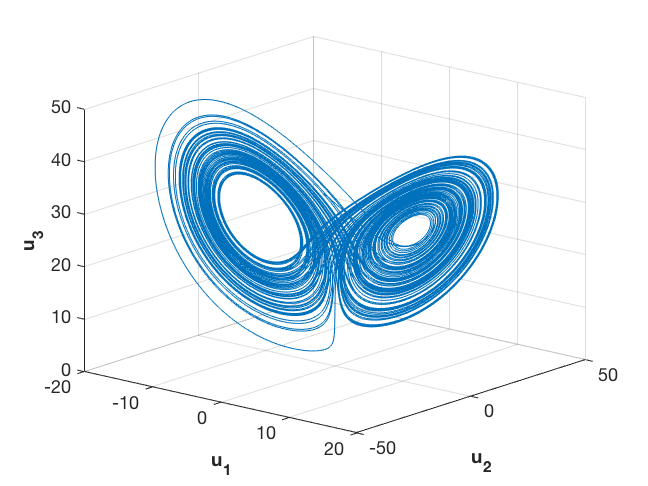}
	}
\subfigure[The trajectory defined by Equation \eqref{eq: lorenz learned2}]{
	\includegraphics[width = 3 in]{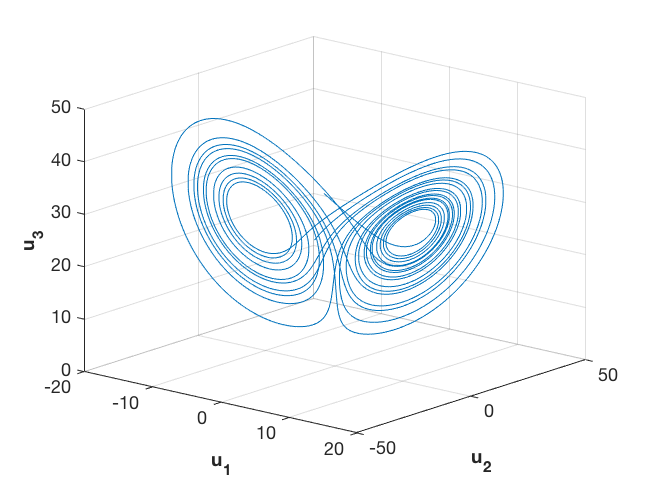}
	\includegraphics[width = 3 in]{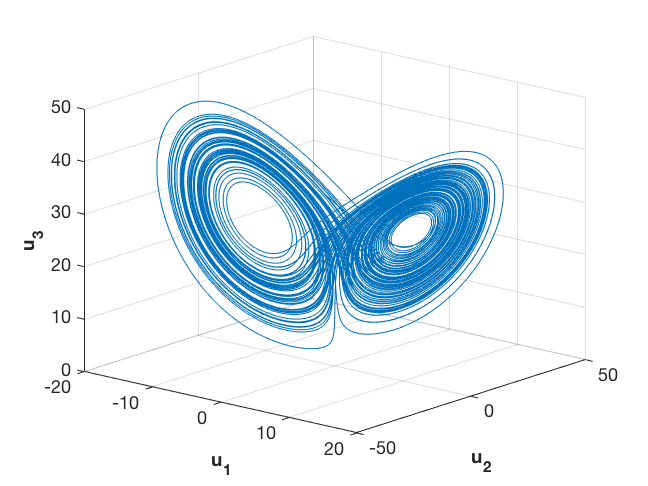}
	}
\caption{\textbf{Lorenz System}: Trajectories of the Lorenz system from $t=0$ to $t=20$ (left column) and from $t=0$ to $t=100$ (right column). (a) The ``true'' Lorenz attractor defined by Equation \eqref{eq: lorenz system}. (b) The trajectory defined by Equation \eqref{eq: lorenz learned1}, which is identified from data with additive noise $\sigma^2=0.1$. (c) The trajectory defined by Equation \eqref{eq: lorenz learned2}, which is identified from data with additive noise $\sigma^2=0.5$. }
\label{fig: lorenz plot3}
\end{figure}

\subsection{The Thomas System}

Consider the Thomas system:
\begin{equation} \label{eq: thomas system}
\begin{cases}
\dot{u}_1 = -0.18 u_1 + \sin(u_2), \\
\dot{u}_2 = -0.18 u_2 + \sin(u_3), \\
\dot{u}_3 = -0.18 u_4 + \sin(u_1).
\end{cases}
\end{equation}
which is a non-polynomial system whose trajectories form a chaotic attractor.  We simulate $u(t)$ using the initial condition $u(0)=\begin{pmatrix} 1, 1, 0 \end{pmatrix}^T$ and by evolving the system using the Runge-Kutta method of order 4 up to time-stamp $t = 100$ with time step $h=0.025$. We then add Gaussian noise to $u$ and obtain the observed noisy data $\tilde u(t)$:
\begin{align*}
\tilde u = u + \eta, \quad \eta \sim \mathcal{N}(0,\sigma^2).
\end{align*}
Let $b$ be the numerical approximation of $\dot u$ which is defined by Equation \eqref{eq: lorenz derivative}. To identify the governing equation for the data generated by the Thomas system ({\it e.g.} the right-hand side of Equation \eqref{eq: thomas system}), we apply the algorithm to the linear system whose dictionary matrix $A=A(\tilde u(t))$ consists of three sub-matrices:
\begin{align*}
A = 
\begin{pmatrix}
| & | & | \\
A_P & A_{\sin} & A_{\cos} \\
| & | & |
\end{pmatrix},
\end{align*}
where
\begin{align*}
A_P &= 
\begin{pmatrix}
| & | & | & | & & | \\
1 & P(\tilde u(t)) & P^2(\tilde u(t)) & P^3(\tilde u(t)) & \cdots & P^{p_1}(\tilde u(t)) \\
| & | & | & | & & | 
\end{pmatrix}, \\
A_{\sin} &= 
\begin{pmatrix}
| & | & | & & | \\
\sin\left(P(\tilde u(t))\right) & \sin\left(P^2(\tilde u(t))\right) & \sin\left(P^3(\tilde u(t))\right) & \cdots & \sin\left(P^{p_2}(\tilde u(t))\right) \\
| & | & | & & |
\end{pmatrix}, \\
A_{\cos} &= 
\begin{pmatrix}
| & | & | & & | \\
\cos\left(P(\tilde u(t))\right) & \cos\left(P^2(\tilde u(t))\right) & \cos\left(P^3(\tilde u(t))\right) & \cdots & \cos\left(P^{p_3}(\tilde u(t))\right) \\
| & | & | & & |
\end{pmatrix}.
\end{align*}
Here, $P^p$ is defined by Equation \eqref{eq: lorenz polynomial}, which denotes the matrix consisting of polynomials in $\tilde u$ of order $p$. The matrices $\sin(P^p)$ and $\cos(P^p)$ are obtained by applying the sine and cosine functions to each element of $P^p$, respectively.

With $p_1=3$, $p_2=p_3=1$, and $\lambda=0.1$, we apply the algorithm to data with different noise levels. The resulting approximations for $x$ are listed in Table \ref{tab: thomas coefficient}. The identified systems are:
\begin{enumerate}[(i)]
\item $\sigma^2=0.1$ (where $\snr(u,\eta)=25.8469$):
\begin{align} \label{eq: thomas learned1}
\begin{cases}
\dot{u}_1 = -0.1805u_1 + 1.0014\sin(u_2) \\
\dot{u}_2 = -0.1799u_2 + 1.0038\sin(u_3) \\
\dot{u}_3 = -0.1803u_3 + 0.9992\sin(u_1)
\end{cases}
\end{align}
with $E(x)=0.0023$;
\item $\sigma^2=0.5$ (where $\snr(u,\eta)=11.8738$):
\begin{align} \label{eq: thomas learned2}
\begin{cases}
\dot{u}_1 = -0.1835u_1 + 1.0304\sin(u_2) \\
\dot{u}_2 = -0.1848u_2 + 0.9956\sin(u_3) \\
\dot{u}_3 = -0.1725u_3 + 0.9658\sin(u_1)
\end{cases}
\end{align}
with $E(x)=0.0267$.
\end{enumerate}

\begin{table}[b!] 
\caption{\textbf{The Thomas system}: The recovered coefficients for two noise levels.} 
\label{tab: thomas coefficient}
\centering \vskip5pt
\begin{tabular}{c | ccc | ccc} \toprule
\multirow{2}{*}{$A$} & \multicolumn{3}{c|}{$\sigma^2=0.1$} & \multicolumn{3}{c}{$\sigma^2=0.5$} \\ \cline{2-7}
& {$\dot{u}_1$} & {$\dot{u}_2$} & {$\dot{u}_3$} & {$\dot{u}_1$} & {$\dot{u}_2$} & {$\dot{u}_3$} \\ \toprule
{$1$} 	& 0	& 0	& 0	& 0	& 0	& 0  \\ \midrule
{$u_1$} 	& -0.1805	& 0	& 0	& -0.1835 & 0 & 0 \\ \midrule
{$u_2$}	& 0	& -0.1799 & 0	& 0	& -0.1848	& 0 \\ \midrule
{$u_3$}	& 0	& 0	& -0.1803	& 0	& 0	& -0.1725 \\ \midrule
{$u_1^2$} & 0	& 0	& 0	& 0	& 0	& 0 \\ \midrule
{$u_1u_2$}& 0 	& 0	& 0	& 0	& 0	& 0 \\ \midrule
{$u_1u_3$}& 0 	& 0	& 0	& 0	& 0	& 0 \\ \midrule
{$u_2^2$}	& 0	& 0	& 0	& 0	& 0	& 0 \\ \midrule
{$\vdots$} & {$\vdots$} & {$\vdots$} & {$\vdots$} & {$\vdots$} & {$\vdots$} &{$\vdots$}  \\ \midrule
{$u_3^3$} & 0	& 0	& 0	& 0	& 0	& 0 \\ \midrule
{$\sin(u_1)$} & 0 & 0 & 0.9992 & 0 & 0 & 0.9658 \\ \midrule
{$\sin(u_2)$} & 1.0014 & 0 & 0 & 1.0304 & 0 & 0 \\ \midrule
{$\sin(u_3)$} & 0 & 1.0038 & 0 & 0 & 0.9956 & 0 \\ \midrule
{$\cos(u_1)$} & 0	& 0	& 0	& 0	& 0	& 0 \\ \midrule
{$\cos(u_2)$} & 0	& 0	& 0	& 0	& 0	& 0 \\ \midrule
{$\cos(u_3)$} & 0	& 0	& 0	& 0	& 0	& 0 \\ \bottomrule
\end{tabular}\end{table}

Observe that the identified system defined by Equation \eqref{eq: thomas learned1} is exact up to two significant digits. We simulate this system up to time-stamps $t=200$ and $t=1000$ and compare it with the trajectories of the Thomas system. We show the short-time evolution of the trajectories in Figure \ref{fig: thomas plot2} and the long-time dynamics in Figure \ref{fig: thomas plot3}. It can be observed that although the coefficients are not exact, the trajectory of the identified system traces out a similar region to the exact trajectory in state-space.

\begin{figure}[h!]
\centering
\includegraphics[width = 6 in]{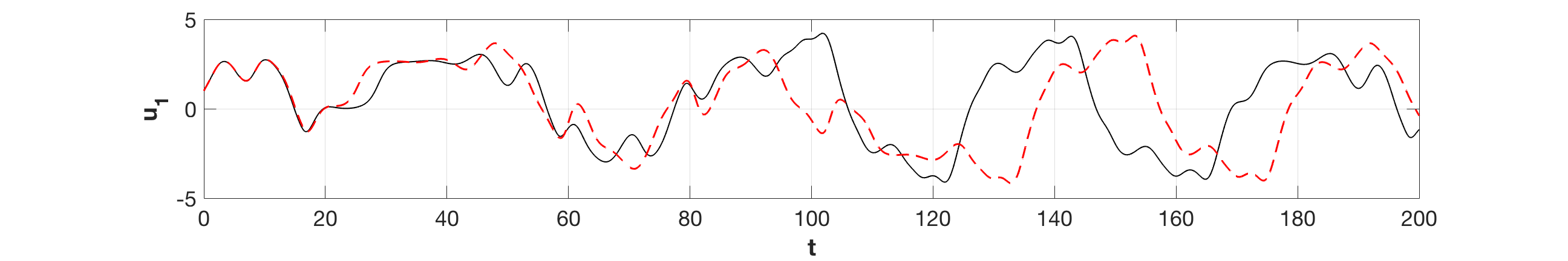}
\includegraphics[width = 6 in]{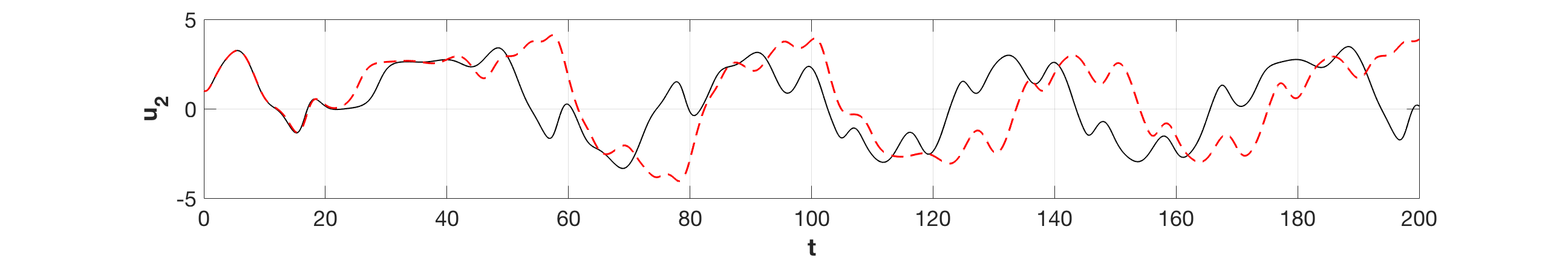}
\includegraphics[width = 6 in]{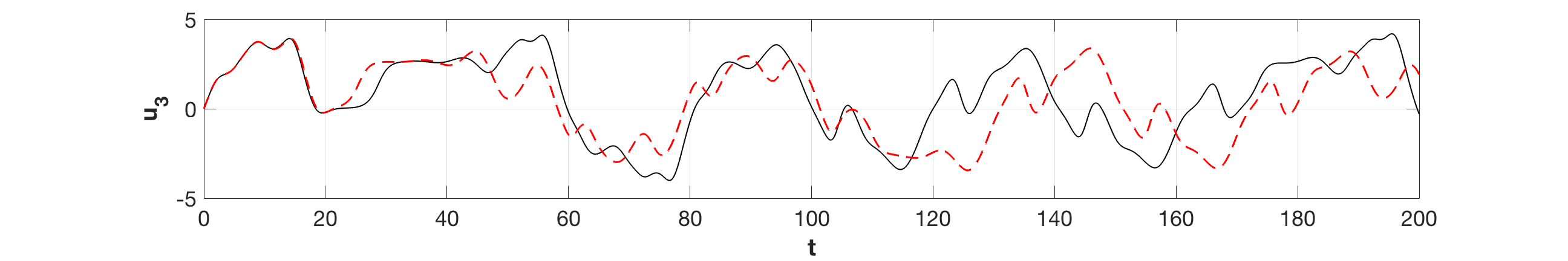}
\caption{\textbf{The Thomas system}: Component-wise evolution of the trajectories. Solid line: the trajectory of the identified system defined by Equation \eqref{eq: thomas learned1}. Red dashed line: the ``true'' Thomas trajectory.}
\label{fig: thomas plot2}
\end{figure}

\begin{figure}[h!]
\centering
\subfigure[The Thomas trajectory]{\label{fig: thomas exact}
	\includegraphics[width = 3 in]{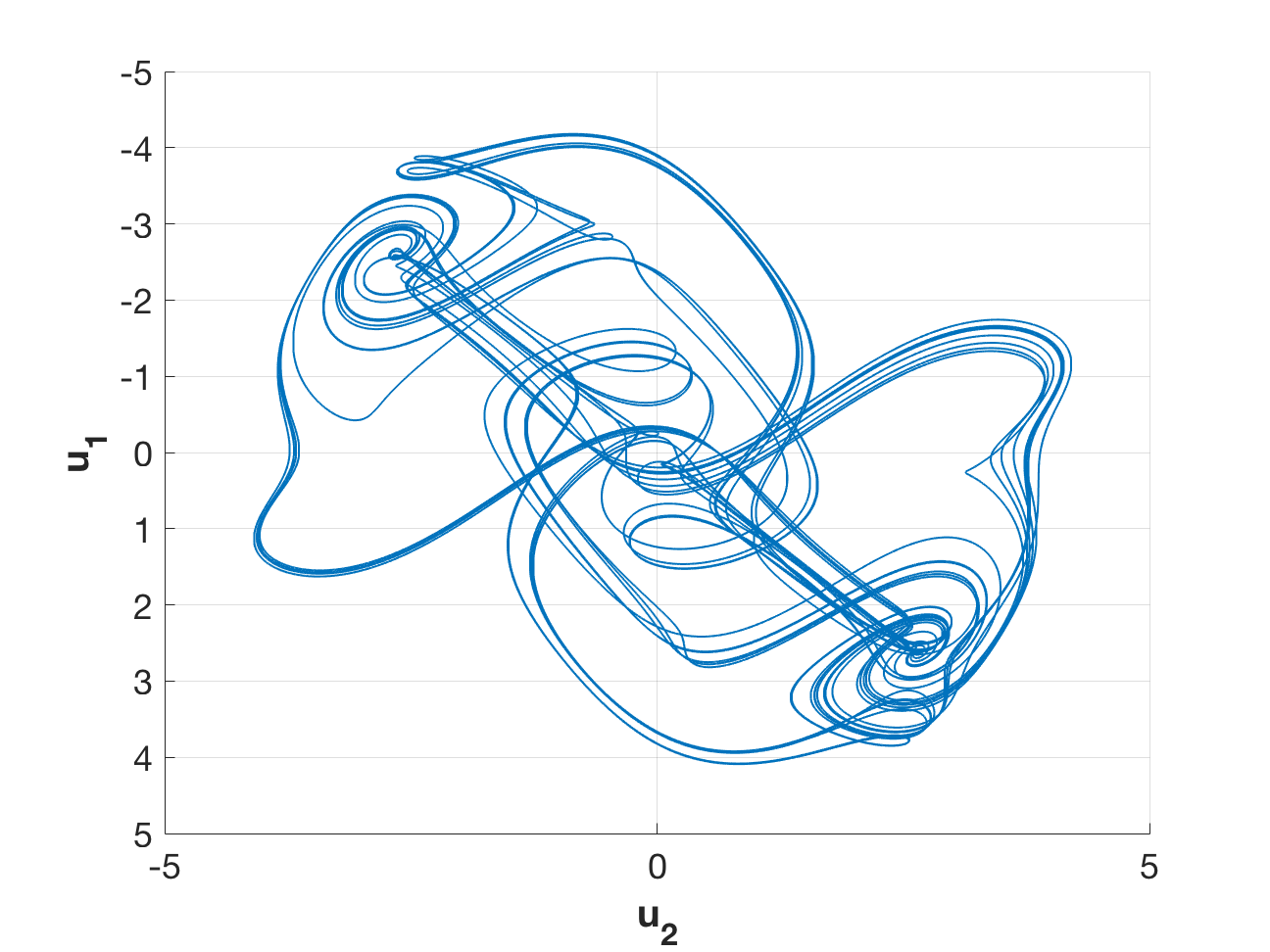}
	}
\subfigure[The trajectory defined by Equation \eqref{eq: thomas learned1}]{
	\includegraphics[width = 3 in]{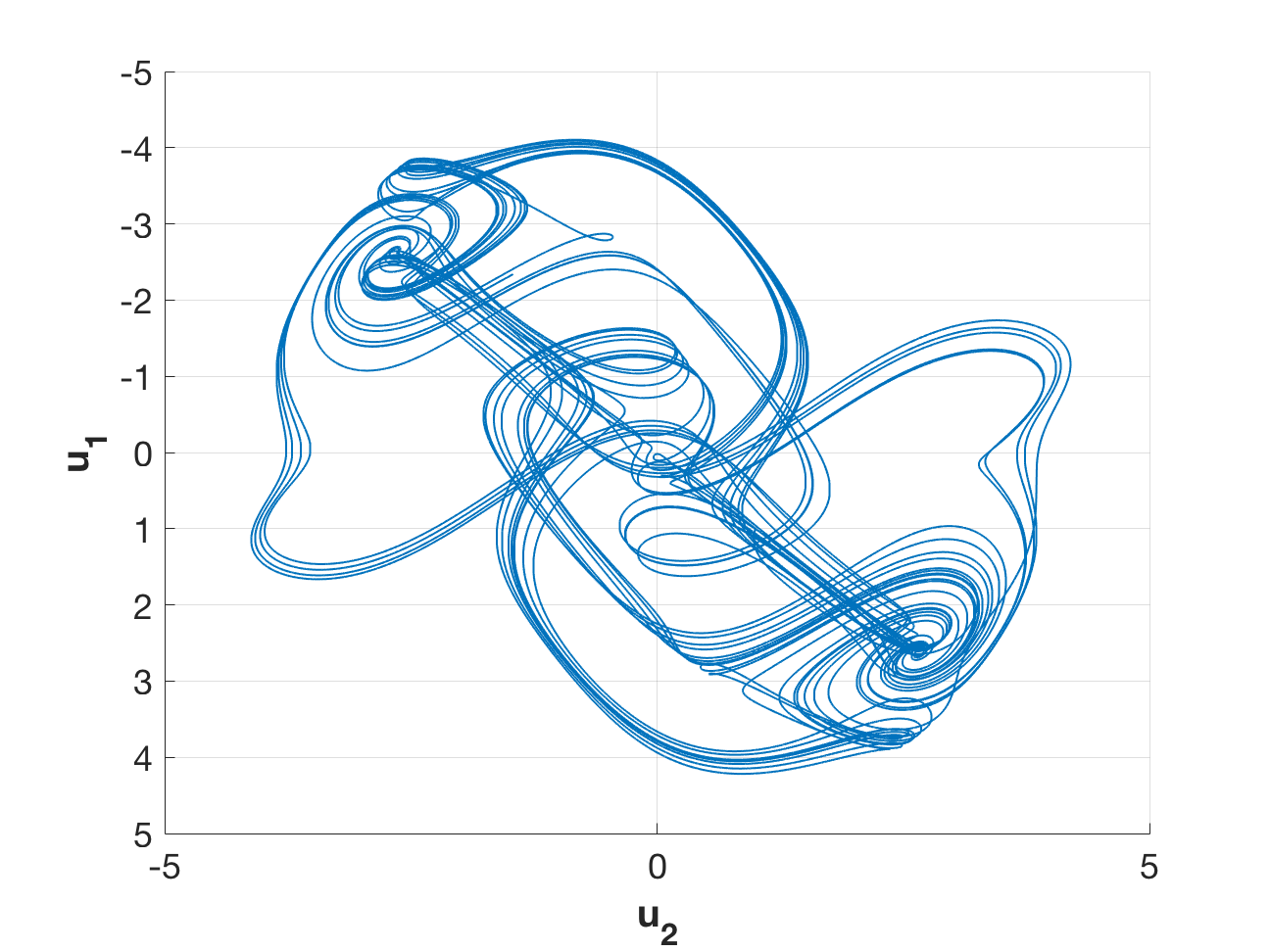}
	}
\caption{\textbf{The Thomas system:} Trajectories of the learned and ``true'' Thomas system from $t=0$ to $t=1000$ (right column). (a) The Thomas trajectory defined by Equation \eqref{eq: thomas system}. (b) The trajectory defined by Equation \eqref{eq: thomas learned1}, which is identified from data with additive noise $\sigma^2=0.1$. }
\label{fig: thomas plot3}
\end{figure}

\section{Discussion}
\label{sec: discussion}

The SINDy algorithm proposed in \cite{BruntonProctorKutz16PNAS} has been applied to various problems involving sparse model identification from complex dynamics. In this work, we provided several theoretical results that characterized the solutions produced by the algorithm and provided the rate of convergence of the algorithm. The results included showing that the algorithm approximates local minimizers of the $\ell^0$-penalized least-squares problem, and thus can be characterized through various sparse optimization results. In particular, the algorithm produces a  minimizing sequence, which converges to a fixed-point rapidly, thereby providing theoretical support for the observed behavior. Several examples show that the convergence rates are sharp. In addition, we showed that iterating the steps is required, in particular, it is possible to obtain solutions through iterating that cannot be obtain via thresholding of the least-squares solution. In future work, we would like to better characterize the effects of noise, detailed in Section~\ref{sec: application}. It would be useful to have a quantifiable relationship between the thresholding parameter, the noise, and the expected recovery error.

\section*{Acknowledgement}

The authors would like to thank J. Nathan Kutz for helpful discussions. H.S. and L.Z. acknowledge the support of AFOSR, FA9550-17-1-0125 and the support of NSF CAREER grant $\#1752116$.

\bibliography{algorithmproof_ref}
\bibliographystyle{plain}

\end{document}